\documentclass{amsart}[12pt]
\usepackage{amsmath, amsthm, amscd, amsfonts,}
\usepackage{graphicx,float}

\usepackage{amssymb}
\usepackage{pb-diagram}
\usepackage{lamsarrow}
\usepackage{pb-lams}

\newtheorem{theorem}{Theorem}[section]

\newtheorem{question}[theorem]{Question}
\numberwithin{equation}{section}

\usepackage{pb-diagram}

\def\rmark{\mbox{$\rm\bf\rule{0.06em}{1.45ex}\kern-0.05em R$}}
\def\pmark{\mbox{$\rm\bf\rule{0.06em}{1.45ex}\kern-0.05em P$}}
\def\nmark{\mbox{$\rm\bf\rule{0.06em}{1.45ex}\kern-0.05em N$}}
\def\vdash{\mbox{$\rm\| \kern-0.13em -$}}

\def\rmark{\mbox{$\rm\bf\rule{0.06em}{1.45ex}\kern-0.05em R$}}
\def\pmark{\mbox{$\rm\bf\rule{0.06em}{1.45ex}\kern-0.05em P$}}
\def\nmark{\mbox{$\rm\bf\rule{0.06em}{1.45ex}\kern-0.05em N$}}
\def\vdash{\mbox{$\rm\| \kern-0.13em -$}}

\begin{document}

\title[On a question of Zadrozny]{On a question of Zadrozny}

\author[ M. Golshani.]{Mohammad Golshani}

\thanks{The author's research has been supported by a grant from IPM (No. 91030417).}
\maketitle

In this short note, I will discuss the following question of Zadrozny \cite{Zadrozny}.

\begin{question} Assume $0^\sharp$ does not exists. Let $M$ be a model of $ZFC$. Is there a model $N$ of $ZFC$ extending $M$
with $ON^N=ON^M$ such that $HOD^N=L?$
\end{question}

Our next result gives a negative answer to this question in a strong way.
\begin{theorem}
Any model $V$ of $ZFC$ has a class generic extension $V'\models$``$ZFC$'' such that if $W \supseteq V'$ is a model of $ZFC$ with
$ON^W=ON^{V'},$ then $HOD^W \neq V.$
\end{theorem}
\begin{proof}
Force over $V$ by the reverse Easton iteration to add a new Cohen subset to each successor cardinal. Call the resulting extension $V_1$.
By Jensen's coding theorem, $V_1$ has a class generic extension $V_2$ such that for some real $R \in V_2,$
we have $V_2\models$``$V=L[R]$''. We show that $V'=V_2$
is as required.  Thus suppose $W \supseteq V'$ is a model of $ZFC$, and suppose on the contrary that $HOD^W=V.$ Then as $R\in W,$
by a result of Vopenka, $R$ is set generic over $HOD^W=V,$ and this is a contradiction.
\end{proof}

School of Mathematics, Institute for Research in Fundamental Sciences (IPM), P.O. Box:
19395-5746, Tehran-Iran.

E-mail address: golshani.m@gmail.com

\end{document}